\documentclass[11pt,reqno]{amsart}

\usepackage{amsmath, amsthm, amssymb}
\usepackage{amsfonts}
\usepackage[ansinew]{inputenc}
\usepackage[dvips]{epsfig}
\usepackage{graphicx}
\usepackage[english]{babel}
\usepackage{cite}
\usepackage{graphicx}
\usepackage{amscd}
\usepackage{xcolor}
\usepackage{bm}
\usepackage{enumerate}

\usepackage{verbatim}
\usepackage{hyperref}
\usepackage{amstext}
\usepackage{latexsym}
\let\oldsqrt\sqrt
\def\sqrt{\mathpalette\DHLhksqrt}
\def\DHLhksqrt#1#2{%
\setbox0=\hbox{$#1\oldsqrt{#2\,}$}\dimen0=\ht0
\advance\dimen0-0.2\ht0
\setbox2=\hbox{\vrule height\ht0 depth -\dimen0}%
{\box0\lower0.4pt\box2}}

\allowdisplaybreaks

\newcommand{\R}{\mathbb{R}} % reelle Zahlen
\newcommand{\N}{\mathbb{N}} % natuerliche Zahlen
\newcommand{\dist}{\textnormal{dist}} % dist ...
\newcommand{\supp}{\textnormal{supp}} % supp ...
\newcommand{\ov}{\overline}

\renewcommand{\k}{\mathbf{k}}

\renewcommand{\H}{{\mathcal H}}

\renewcommand{\phi}{\varphi}

\newcommand{\rz}{\mathbb{R}}

\newcommand{\cB}{{\mathcal B}}
\newcommand{\cC}{{\mathcal C}}

\newcommand{\cE}{{\mathcal E}}

\newcommand{\cH}{{\mathcal H}}

\newcommand{\cL}{{\mathcal L}}

\newcommand{\cR}{{\mathcal R}}

\newcommand{\cV}{{\mathcal V}}

\newcommand{\B}{{\mathcal B}}

\newcommand{\eps}{\varepsilon}

\theoremstyle{definition}
\newtheorem{defi}{Definition}[section]

\theoremstyle{plain} %default%plain
\newtheorem{thm}[defi]{Theorem}
\newtheorem{prop}[defi]{Proposition}
\newtheorem{lemma}[defi]{Lemma}

\newtheorem{cor}[defi]{Corollary}

\theoremstyle{definition}

\numberwithin{equation}{section}

 \title{Morse index versus radial symmetry for fractional Dirichlet problems}

\author[]
{Mouhamed Moustapha Fall, Pierre Aime Feulefack, Remi Yvant Temgoua, Tobias Weth}

\address{African Institute for Mathematical Sciences in Senegal (AIMS Senegal), 
KM 2, Route de Joal, B.P. 14 18. Mbour, S\'en\'egal.}

\email{mouhamed.m.fall@aims-senegal.org}

\address{Goethe-Universit\"{a}t Frankfurt, Institut f\"{u}r Mathematik.
Robert-Mayer-Str. 10, D-60629 Frankfurt, Germany.}

\email{feulefac@math.uni-frankfurt.de}

\address{Goethe-Universit\"{a}t Frankfurt, Institut f\"{u}r Mathematik.
Robert-Mayer-Str. 10, D-60629 Frankfurt, Germany.}

\email{temgoua@math.uni-frankfurt.de}

\address{Goethe-Universit\"{a}t Frankfurt, Institut f\"{u}r Mathematik.
Robert-Mayer-Str. 10, D-60629 Frankfurt, Germany.}

\email{weth@math.uni-frankfurt.de}

\date{\today}

\begin{document}
\maketitle
\begin{abstract}
	In this work, we provide an estimate of the Morse index of radially symmetric sign changing bounded weak solutions $u$
 to the semilinear fractional Dirichlet problem 
$$		
(-\Delta)^su = f(u)\qquad \text{ in $\B$},\qquad \qquad u =  0\qquad \text{in $\quad\rz^{N}\setminus \B$,}
$$
where $s\in(0,1)$,  $\B\subset \R^N$ is the unit ball centred at zero  and the nonlinearity  $f$ is of class $C^1$. We prove that  for $s\in(1/2,1)$  any radially  symmetric  sign changing  solution of the above problem has a Morse index greater than or equal to $N+1$. If $s\in (0,1/2],$ the same conclusion holds under additional assumption on $f$. In particular, our results apply to the Dirichlet eigenvalue problem for the operator $(-\Delta)^s$ in $\B$ for all $s\in (0,1)$, and imply that eigenfunctions corresponding to the second Dirichlet eigenvalue in $\B$ are antisymmetric. This resolves a conjecture of Ba\~{n}uelos and Kulczycki.  
\end{abstract}

{\footnotesize
\begin{center}
\end{center}
}

\section{Introduction and main result}\label{intro}%%%%%%%%%%%%%%%%%%%%%%%%%%%%%%%%%

 The purpose of this paper is to estimate the Morse index of radial sign changing solutions of the problem
	\begin{equation}\label{eq1}
	\quad\left\{\begin{aligned}
		(-\Delta)^su &= f(u) && \text{ in\ \  $\B$}\\
		u &=  0             && \text{ in\ \  }\rz^{N}\setminus \B,	 
	\end{aligned}\right.
	\end{equation}
	 where $s\in(0,1)$, $\B\subset \R^N$ is the unit ball centred at zero and where the nonlinearity $f: \R \rightarrow \R$ is of class $C^1$. The fractional Laplacian operator $(-\Delta)^s$ is defined for all $u\in C^{2}_c(\R^N)$ by
\begin{equation}\nonumber
 (-\Delta)^s u(x)=c(N,s)\lim_{\epsilon\to 0^+}\int_{\R^N\setminus B_{\epsilon}(x)}\frac{u(x)-u(y)}{|x-y|^{N+2s}}\ dy,
\end{equation}
where $c(N,s)= 2^{2s}\pi^{-\frac N2}s\frac{\Gamma(\frac{N+2s}2)}{\Gamma(1-s)}$  is a  normalization constant. % see   \cite{bisci2016variational, di2012hitchhiker} and also for suitable properties on fractional Sobolev spaces.\\
The operator $(-\Delta)^s$ can be seen as the infinitesimal generator of an isotropic stable L\'{e}vy processes (see \cite{applebaum2004levy}), and it arises in specific mathematical models within several areas of physics, biology, chemistry and finance (see \cite{applebaum2004levy,applebaum2009levy,bucur-valdinoci}). For basic properties of $(-\Delta)^s$ and associated function spaces, we refer to \cite{di2012hitchhiker}. 

In recent years, the study of linear and nonlinear Dirichlet boundary value problems involving  fractional Laplacian has attracted extensive and steadily growing attention, whereas, in contrast to the local case $s=1$, even basic questions
still remain largely unsolved up to now. Even in the linear case where $f(t):= \lambda t$, the structure of Dirichlet eigenvalues and eigenfunctions of the fractional Laplacian on the unit ball $\B$ is not completely understood. In particular, we mention a conjecture of Ba\~{n}uelos and Kulczycki which states that every Dirichlet eigenfunction $u$ of $(-\Delta)^s$ on $\B$ corresponding to the second Dirichlet eigenvalue is antisymmetric, i.e., it satisfies $u(-x)=-u(x)$ for $x \in \B$. So far, by the results in \cite{ban-kulc,kwas,RF19,dyda2017eigenvalues}, this conjecture has been verified in the special cases $N \le 3$, $s \in (0,1)$ and $4 \le N \le 9$, $s= \frac{1}{2}$. In the present paper, we will derive the full conjecture essentially as a corollary of our main result on the semilinear Dirichlet problem (\ref{eq1}), see Theorem~\ref{symmetry-of-higher-eigenfunctions} below.

Our main result on sign changing radial solutions of (\ref{eq1}) is heavily inspired by the seminal work of Aftalion and Pacella \cite{aftalion2004qualitative}, where the authors studied qualitative properties of sign changing solutions of the local semilinear elliptic problem
\begin{equation}\label{eq01}
-\Delta u = f(u)\ \ \ \ \text{in}\ \ \Omega,\ \ \ \  u = 0\ \ \text{on}\ \ \ \ \partial\Omega,
\end{equation}
where $\Omega \subset \R^N$ is a ball or an annulus centered at zero and $f\in C^1(\R)$. It is proved in \cite[Theorem 1.1]{aftalion2004qualitative} that any radial sign changing solution of \eqref{eq01} has Morse index greater than or equal to $N+1.$

In the following, we present a nonlocal version of this result in the case where $\Omega$ is the unit ball in $\R^N$. We need to fix some notation first. Consider the function space
\begin{equation}\label{wsp-d-special-case}
\cH^{s}_{0}(\B):=\{u\in H^{s}(\R^N) : u\equiv 0 ~\text{ on\ \ $\R^{N}\setminus \B$}\} \subset H^{s}(\R^N).
\end{equation}
By definition, a function $u \in \cH^s_0(\B) \cap L^\infty(\B)$ is a weak solution of (\ref{eq1}) if 
$$
\cE_s(u,v) = \int_{\B} f(u)v\,dx \qquad \text{for all $v \in \cH^s_0(\B)$,}
$$
where 
\begin{equation}\label{bilinear form}
(v,w) \mapsto \cE_s(v,w):=\frac{c(N,s)}{2}\int_{\R^N}\int_{\R^N}\frac{(v(x)-v(y))(w(x)-w(y))}{|x-y|^{N+2s}}\ dxdy.
\end{equation}
is the bilinear form associated with $(-\Delta)^s$. By definition, the \textit{Morse index} $m(u)$ of a weak solution $u \in \cH^s_0(\B) \cap L^\infty(\B)$ of \eqref{eq1} is the maximal dimension of a subspace $X \subset\cH^s_0(\B)$ where the quadratic form 
\begin{equation}\label{bilinear-form-L}
(v,w) \mapsto \cE_s(v,w)-\int_{\B}f'(u)vw\,dx 
\end{equation}
associated to the linearized operator $L:=(-\Delta)^s-f'(u)$ is negative definite. Equivalently, $m(u)$ can be defined as the number of the negative Dirichlet eigenvalues of $L$ counted with their multiplicity. 

Our first main result reads as follows.  
 \begin{thm}\label{T}
Let $u$ be a radially symmetric sign changing solution of problem \eqref{eq1}, and suppose that \underline{one} of the following additional conditions holds.
 \begin{itemize}
 \item[(A1)] $s \in (\frac{1}{2},1)$.
 \item[(A2)] $s \in (0,\frac{1}{2}]$, and 
   \begin{equation}
     \label{eq:add-cond}
 \int_0^t f(\tau)d\tau > \frac{N-2s}{2N}\, t f(t) \qquad \text{for $t \in \R \setminus \{0\}$.}       
   \end{equation}
 \end{itemize}
Then $u$ has Morse index greater than or equal to $N+1$. 
 \end{thm}
 
 We briefly comment on the inequality (\ref{eq:add-cond}).  In our proof of Theorem~\ref{T}, this assumption arises when we use the Pohozaev identity for the fractional Laplacian, see \cite[Theorem 1.1]{Ros14-1}. It is satisfied for homogeneous nonlinearities with subcritical growth, i.e., if 
$$
f(t)= \lambda |t|^{p-2}t\qquad \text{with $\lambda >0$ and $2 \le p < \frac{2N}{N-2s}$.}
$$
We also note that, in the supercritical case where $\int_0^t f(\tau)d\tau < \frac{N-2s}{2N}\, t f(t)$ for $t \in \R \setminus \{0\}$, problem (\ref{eq1}) does not admit any nontrivial weak solutions $u \in \cH^s_0(\B) \cap L^\infty(\B)$. This is a consequence of the Pohozaev identity stated in \cite[Theorem 1.1]{Ros14-1}.

In particular, assumption~(\ref{eq:add-cond}) is satisfied in the linear case $t \mapsto \lambda t$ with $\lambda >0$. In fact, we can deduce the following result for the Dirichlet eigenvalue problem 
	\begin{equation}\label{eq1-linear}
	\quad\left\{\begin{aligned}
		(-\Delta)^su &= \lambda u && \text{ in\ \  $\B$}\\
		u &=  0             && \text{ in\ \  }\rz^{N}\setminus \B,	 
	\end{aligned}\right.
	\end{equation}
from Theorem~\ref{T}, thereby providing a complete positive answer to a conjecture by Ba\~{n}uelos and Kulczycki (see \cite{dyda2017eigenvalues}).  

 \begin{thm}\label{symmetry-of-higher-eigenfunctions}  
 Let $N \ge 1$ and $0<s<1$, and let $\lambda_2 >0$ be the second eigenvalue of problem~(\ref{eq1-linear}). Then every eigenfunction $u$ corresponding to $\lambda_2$ is antisymmetric, i.e. it satisfies $u(-x)=-u(x)$ for $x \in \B$.
\end{thm}

In recent years, partial results towards this conjecture have been obtained in  \cite{ban-kulc,kwas,RF19,dyda2017eigenvalues}, covering the special cases $N \le 3$, $s \in (0,1)$ and $4 \le N \le 9$, $s= \frac{1}{2}$. More precisely, in \cite[Theorem 5.3]{ban-kulc}, Ba\~{n}uelos and Kulczycki proved antisymmetry of second eigenfunctions in the special case $N=1$, $s=\frac{1}{2}$. In \cite{kwas}, this result was extended to $N=1$, $s \in [\frac{1}{2},1)$. Recently in \cite{dyda2017eigenvalues}, the conjecture was proved in the cases $N \le 2$, $s \in (0,1)$ and $3 \le N \le 9$, $s = \frac{1}{2}$. Moreover, in \cite{RF19}, the result has been proved for $N=3$, $s \in (0,1)$. 

While the proofs in these papers are based on fine eigenvalue estimates, our proof of Theorem~\ref{symmetry-of-higher-eigenfunctions} is completely different: In addition to Theorem~\ref{T}, we shall only use the following important alternative which is implicitely stated in \cite[p. 503]{dyda2017eigenvalues}: {\em Either (\ref{eq1-linear}) admits a radially symmetric eigenfunction corresponding to the second eigenvalue $\lambda_2$, or every eigenfunction corresponding to $\lambda_2$ is a product of a linear and a radial function.} Since every such eigenfunction $u$ is a sign changing solution of (\ref{eq1}) with $t \mapsto f(t)= \lambda_2 t$ and has Morse index $1 < N+1$, it cannot be radially symmetric as a consequence of Theorem~\ref{T}. Hence $u$ must be a product of a linear and a radial function, and therefore $u$ is antisymmetric. This completes the proof of Theorem~\ref{symmetry-of-higher-eigenfunctions}. For a more detailed presentation of this argument and the underlying results from \cite{dyda2017eigenvalues}, see Section~\ref{alternative} below.

We briefly comment on the proof of Theorem~\ref{T}. The general strategy, inspired by the paper \cite{aftalion2004qualitative} of Aftalion and Pacella for the local problem (\ref{eq01}), is to use partial derivatives of $u$ to construct suitable test functions which allow to estimate the Morse index of $u$. In the nonlocal case, several difficulties arise since local PDEs techniques do not apply. The most severe difficulty is related to the fact that weak solutions $u \in \cH^s_0(\B) \cap L^\infty(\B)$ of (\ref{eq1}) have much less boundary regularity than solutions of (\ref{eq01}), see Proposition~\ref{reg-prop-solutions} for details. Moreover, even though there exists a fractional version of the Hopf boundary lemma related to the fractional boundary derivative $\frac{u}{\delta^s}$ (see \cite[Proposition 3.3]{fall-jarohs-2015}), it does not apply to sign changing solutions of (\ref{eq1}) due to the non-locality of the problem. We mention 
 at this point that the classical Hopf boundary lemma is used in \cite{aftalion2004qualitative} together with an extra assumption on $f(0)$, but a slight change of the proof, exploiting the local character of the problem, allows to deal with solutions $u$ having a vanishing derivative on the boundary; therefore \cite[Theorem 1.1]{aftalion2004qualitative} extends to arbitrary nonlinearities $f \in C^1(\R)$\footnote{We wish to thank the referee for pointing out this fact.}. In the nonlocal case of radial solutions $u$ of (\ref{eq1}), it is more difficult to deal with possible oscillations of the radial derivative of $u$ close to the boundary. In our proof of Theorem~\ref{T}, we distinguish two cases. In the case $s \in (\frac{1}{2},1)$, we use a regularity result of Grubb given in \cite[Theorem 2.2]{Gr14} to complete the argument in the case where $\frac{u}{\delta^s}$ vanishes on $\partial \B$. Moreover, in the case $s \in (0,\frac{1}{2}]$, we use the extra assumption (\ref{eq:add-cond}) to ensure that $\frac{u}{\delta^s}$ does not vanish on the boundary. Here we point out that (\ref{eq:add-cond}) implies $f(0)=0$, while no extra assumption on $f(0)$ is needed in the case $s \in (\frac{1}{2},1)$.

We point out that our proof of Theorem~\ref{T} does not use the extension method of Caffarelli and Silvestre \cite{caffarelli2007extension}, which allows to reformulate (\ref{eq1}) as a boundary value problem where 
$(-\Delta)^s$ arises as a Dirichlet-to-Neumann type operator. We therefore expect that our approach applies to a more general class of nonlocal operators in place of $(-\Delta)^s$.

We wish to add some remarks on the role of Morse index estimates in the variational study of (\ref{eq1}). In the case where 
$f\in C^1(\R)$ has subcritical growth, weak solutions of (\ref{eq1}) are precisely the critical points of the associated energy functional $J: \cH^s_{0}(\B)\to\R$ defined by
$$
J(u) =  \frac{c(N,s)}{2}\displaystyle\int_{\R^N}\int_{\R^N}\frac{|u(x)-u(y)|^2}{|x-y|^{N+2s}}\ dxdy - \int_{\B} F(u)\ dx,
$$
where $F(t) = \int_{0}^tf(s)\ ds.$ Moreover, $J$ is of class $C^2$, and thus the behaviour of $J$ near a critical point $u$ is closely 
related to the Morse index $m(u)$. Typically, critical points detected via minimax principles lead to bounds on the Morse index. In combination with Theorem~\ref{T}, this allows to show the non-radiality of certain classes of sign changing critical points. In this spirit, it is proved in \cite{aftalion2004qualitative} that, under suitable additional assumptions on $f$, least energy sign changing solutions of the local problem (\ref{eq01}) are non-radial functions. 

With regard to the existence of least energy sign changing solutions of the nonlocal problem~\eqref{eq1}, we refer to the recent paper \cite{teng-wang-wang}. For existence results for sign changing solutions to related nonlocal problems, see e.g. \cite{wang2016radial,luo2018sign} and the references therein.\\ 

The paper is organized as follows. In Section \ref{section2} we introduce preliminary notions and collect preliminary results on function spaces. In Section \ref{section3}, we investigate radial solutions of (\ref{eq1}) and properties of their partial derivatives. In Section \ref{section4}  we complete the proof of Theorem  \ref{T}. Finally, in Section~\ref{alternative}, we complete the proof of Theorem~\ref{symmetry-of-higher-eigenfunctions}.\\
 
 \textbf{Acknowledgements:} This work is supported by DAAD and BMBF (Germany) within the project 57385104. Mouhamed Moustapha Fall's work is also supported by the Alexander von Humboldt foundation. The authors would like to thank Xavier Ros-Oton and Sven Jarohs for helpful discussions. Moreover, they would like to thank the referee for valuable comments and suggestions.

\section{Preliminary definitions and results}\label{section2}%%%%%%%%%%%%%%%%%%%%%%%%%%%%%

In this section, we introduce some notation and state preliminary results to be used throughout this paper.

We first introduce and recall some notation related to sets and functions. If $\Omega_1, \Omega_2 \subset \R^N$ are open subsets, we write $\Omega_1 \subset \subset \Omega_2$ if $\overline{\Omega}_1$ is compact and contained in $\Omega_2$. We denote by $1_U:\R^N\rightarrow\R$ the characteristic function of a subset $U\subset\R^N.$ For a function $u:\R^N\rightarrow\R $, we use $u^+:=\max\{u,0\} $ and $u^-:=-\min\{u,0\}$ to denote the positive and negative part of $u$, respectively. 

Next we recall some notation related to function spaces associated with the fractional power $s \in (0,1)$. We consider the space 
\begin{equation}
\cL^1_s:=\bigg\{u\in L^1_{loc}(\R^N):\|u\|_{\cL^1_s}<\infty\bigg\}, \ \ \ \text{where}\ \ \ \|u\|_{\cL^1_s}:=\int_{\R^N}\frac{|u(x)|}{1+|x|^{N+2s}}\ dx.
\end{equation}
If $w \in \cL^1_s$, then $(-\Delta)^s w$ is well defined as a distribution on $\R^N$ by setting 
$$
[(-\Delta)^s w](\phi) = \int_{\R^N} w (-\Delta)^s \phi\,dx \qquad \text{for $\phi \in \cC^\infty_c(\R^N)$.}
$$
Here and in the following, for an open subset $\Omega \subset \R^N$, we denote by $\cC^\infty_c(\Omega)$ the space of smooth functions on $\R^N$ with compact support in $\Omega$. We recall a maximum principle for the fractional Laplacian 
in distributional sense due to Silvestre. 

\begin{prop}\cite[Proposition 2.17]{Ls07}\label{prop1}
	Let $\Omega\subset \R^N$ be an open bounded set, and let $w \in \cL^1_s$ be a lower-semicontinuous function in $\overline{\Omega}$ such that $w\geq0$ in $\R^N\setminus\Omega$ and $(-\Delta)^sw\geq0$ in $\Omega$ in distributional sense, i.e., 
$$
\int_{\R^N}w (-\Delta)^s \phi\,dx \ge 0 \qquad \text{for all nonnegative functions $\phi \in \cC^\infty_c(\Omega)$.}
$$

Then $w\geq0$ in $\R^N.$
\end{prop}

For an open subset $\Omega \subset \R^N$, we now consider the fractional Sobolev space 
\begin{equation}\label{wsp-rn}
H^{s}(\Omega)=\Bigg\{u\in L^{2}(\Omega)\;:\; \int_{\Omega}\int_{\Omega}\frac{|u(x)-u(y)|^2}{|x-y|^{N+2s}}\ dxdy<\infty\Bigg\}.
\end{equation}
Setting 
$$
[u]_{s,\Omega}:= \Bigl(\frac{1}{2}\int_{\Omega}\int_{\Omega}\frac{|u(x)-u(y)|^2}{|x-y|^{N+2s}}\ dxdy\Bigr)^{\frac{1}{2}} \qquad \text{for $u \in H^s(\Omega)$,}
$$
we note that $H^s(\Omega)$ is a Hilbert space whose norm can be written as 
\begin{equation}\label{N1}
\displaystyle\|u\|_{H^s(\Omega)} = \Bigl(\|u\|^{2}_{L^2(\Omega)} + [u]_{s,\Omega}^2\Bigr)^{\frac{1}{2}} 
\end{equation}
We will also use the local fractional Sobolev space $H^s_{loc}(\Omega)$ defined as the space of functions 
$\psi \in L^2_{loc}(\Omega)$ with $\psi\in H^s(\Omega')$ for every $\Omega'\subset \subset \Omega$. 

For a {\em bounded} open subset $\Omega \subset \R^N$, we let $\cH^s_0(\Omega)$ denote the closure of $C^\infty_c(\Omega)$ in $H^s(\R^N)$. Then $\cH^s_0(\Omega)$ is a Hilbert space with scalar product 
\begin{equation}\nonumber
(u,v)\mapsto \cE_s(u,v) :=\langle u,v\rangle_{\cH^s_0(\Omega)}= \frac{c(N,s)}{2}\displaystyle\int_{\R^N}\int_{\R^N}\frac{\big(u(x)-u(y)\big)\big(v(x)-v(y)\big)}{|x-y|^{N+2s}}\ dxdy
\end{equation}
and corresponding norm 
$$
\|u\|_{\cH^{s}_{0}(\Omega)} = \sqrt{\cE_s(u,u)}=  \sqrt{c(N,s)} [u]_{s,\R^N}. 
$$
This is a consequence of the fact that 
$$
\inf \{\cE_s(u,u) \::\: u \in \cH^s_0(\Omega),\: \|u\|_{L^2(\Omega)} = 1\} \;>\;0,
$$
which in turn follows from the fractional Sobolev inequality (see e.g. \cite[Theorem 6.5]{di2012hitchhiker}) and the boundedness of $\Omega$. In particular, $\cH^s_0(\Omega)$ embeds into $L^2(\Omega)$. We also note that, by definition,
\begin{equation}
\label{inclusion-property}
\cH^s_0(\tilde \Omega) \subset \cH^s_0(\Omega) \qquad \text{for bounded open sets $\Omega, \tilde \Omega$ with $\tilde \Omega \subset \Omega$.}
\end{equation}
We also recall the following property, see e.g. \cite[Theorem 1.4.2.2]{G11}: 
\begin{equation}
  \label{eq:density}
  \begin{aligned}
  &\text{\em For any bounded domain $\Omega$ with continuous boundary,}\\
 &\text{\em we have $\cH^{s}_{0}(\Omega)=\{u\in H^{s}(\R^N) : u\equiv 0 ~\text{ on\ \ $\R^{N}\setminus \Omega$}\}.$}  
  \end{aligned}
\end{equation}
Consequently, the definition of $\cH^{s}_{0}(\Omega)$ is consistent with (\ref{wsp-d-special-case}). 

For the remainder of this section, we fix a bounded open subset $\Omega \subset \R^N$. The following lemma is known, but we include a short proof for the convenience of the reader.

\begin{lemma}
\label{compact-support-H-s-0}
Let $\phi \in H_{loc}^s(\Omega)$ be compactly supported in $\Omega$. Then $\phi\in \H^s_{0}(\Omega)$.
\end{lemma}
Here and in the following, we identify $\phi$ with its trivial extension to $\R^N$. 
\begin{proof} 
Without loss of generality, we may assume that $\Omega$ has a continuous boundary, since otherwise we may use (\ref{inclusion-property}) after replacing $\Omega$ by a bounded open subset $\tilde \Omega$ with continuous boundary containing the support of $\phi$. 

Let $\Omega'\subset\subset\Omega$ be an open subset of $\Omega$ which contains the support $K$ of $\phi$. Then we have 
	\begin{equation}\label{ehloc}
\begin{split}
\frac{1}{2}\displaystyle\int_{\R^N}\int_{\R^N}\frac{|\phi(x)-\phi(y)|^2}{|x-y|^{N+2s}}\ dxdy  = [\phi]_{s,\Omega'}^2 + \displaystyle\int_{\Omega'}\int_{\R^N\setminus\Omega'}\frac{|\phi(x)|^2}{|x-y|^{N+2s}}\ dydx,  
\end{split}		
	\end{equation}
where $[\phi]_{s,\Omega'}^2<\infty$ since $\phi\in H^s_{loc}(\Omega)$. Moreover,
\begin{align*}
\int_{\Omega'}\int_{\R^N\setminus\Omega'}\frac{|\phi(x)|^2}{|x-y|^{N+2s}}\ dydx  &= \displaystyle\int_{K}|\phi(x)|^2\ \int_{\R^N\setminus\Omega'}\frac{dy}{|x-y|^{N+2s}}dx\\
&\le \|\phi\|^2_{L^2(K)} \sup_{x \in K}\int_{\R^N\setminus\Omega'}\frac{dy}{|x-y|^{N+2s}} <\infty
\end{align*}
since $\dist(K,\R^N \setminus \Omega')>0$. Since $\Omega$ has a continuous boundary and $\phi\equiv0$ in $\R^N\setminus\Omega,$ we conclude that $\phi\in\cH^s_0(\Omega)$ as a consequence of (\ref{eq:density}).
\end{proof}

We also need the following lemma.

\begin{lemma}
\label{new-lemma-1}
Let $v \in\cL^1_s \cap H^s_{loc}(\Omega)$, and let $\phi \in H^s_{loc}(\Omega)$ be a function with compact support. Then the integral 
$$
\cE_s(v,\phi)= \frac{c(N,s)}{2}\displaystyle\int_{\R^N}\int_{\R^N}\frac{\big(v(x)-v(y)\big)\big(\phi(x)-\phi(y)\big)}{|x-y|^{N+2s}}\ dxdy
$$
is well defined in Lebesgue sense. More precisely, for any choice of open subsets 
$$
\Omega' \subset \subset \Omega'' \subset \subset \Omega
$$
with $\supp \,\phi \subset \Omega'$, there exist constants $c_1,c_2$ -- depending only on $\Omega', \Omega'', N~\text{and}~s$ but not on $v$ and $\phi$ ---such that 
\begin{align}
\frac{1}{2}\int_{\R^N}\int_{\R^N}&\frac{\big|v(x)-v(y)\big|\big|\phi(x)-\phi(y)\big|}{|x-y|^{N+2s}}\ dxdy \label{new-estimate}\\
&\qquad \qquad \le 
[v]_{s,\Omega''} [\phi]_{s,\Omega''} + c_1 \|v\|_{L^2(\Omega')} \|\phi\|_{L^2(\Omega')} + c_2 \|\phi\|_{L^1(\Omega')} \|v\|_{\cL^1_s}. \nonumber
\end{align}
\end{lemma}

\begin{proof}
 We put $\k(z)= |z|^{-N-2s}$. Since $\supp \,\phi \subset \Omega'$, we see that
	\[\begin{split}
	&\frac{1}{2}\int_{\R^N} \int_{\R^N} |v(x)-v(y)||\phi(x)-\phi(y)|\k(x-y)\ dxdy =\\
&\frac{1}{2}\int_{\Omega''} \int_{\Omega''} \frac{\big|v(x)-v(y)\big|\big|\phi(x)-\phi(y)\big|}{|x-y|^{N+2s}}\ dxdy + \int_{\Omega'} \int_{\R^N \setminus \Omega''} \frac{\big|v(x)-v(y)\big|\big|\phi(x)\big|}{|x-y|^{N+2s}}\ dydx\\
&\leq [v]_{s,\Omega''} [\phi]_{s,\Omega''} +
\int_{\Omega'} |\phi(x)| \int_{\R^N\setminus \Omega''} |v(x)-v(y)| \k(x-y)\ dydx,
	\end{split}\]
where
\begin{align*}
\int_{\Omega'} |\phi(x)| &\int_{\R^N\setminus \Omega''} |v(x)-v(y)|\k(x-y)\ dydx\\
&\le \int_{\Omega'}|\phi(x)| |v(x)| \kappa_{\Omega''}(x)\,dx + \int_{\Omega'}
|\phi(x)| \int_{\R^N\setminus \Omega''} |v(y)|\k(x-y) dydx \\
&\le c_1 \|\phi \|_{L^2(\Omega')} \|v\|_{L^2(\Omega')} + c_2 \|\phi\|_{L^1(\Omega')} \|v\|_{\cL^1_s}
\end{align*}
with
$$
\kappa_{\Omega''}(x) = \int_{\R^N \setminus \Omega''}\k(x-y)\ dy,~~~\text{~}~~x\in\Omega' 
$$
and
$$
c_1:= \sup_{x \in \Omega'} \kappa_{\Omega''}(x),  \qquad \quad c_2: = \sup_{x \in \Omega',y \in \R^N \setminus \Omega''} \k(x-y)(1+|y|)^{N+2s}.
$$
Note that the values $c_1$ and $c_2$ are finite since $\Omega' \subset  \subset \Omega''$. It thus follows that $\cE_s(u,v)$ is well-defined in Lebesgue sense and that (\ref{new-estimate}) holds. 
\end{proof}

\begin{cor}
\label{New-cor}
Let $v \in \cL^1_s \cap H^s_{loc}(\Omega)$. If $\Omega' \subset \subset \Omega$ and $(\phi_n)_n$ is a sequence in $H^s_{loc}(\Omega)$ with $\supp\, \phi,\: \supp\, \phi_n \subset \Omega'$ for all $n \in \N$ and $\phi_n \to \phi$ in $H^s_{loc}(\Omega)$, then we have 
$$
\cE_s(v,\phi_n) \to \cE_s(v,\phi) \qquad \text{as $n \to \infty$.}
$$
\end{cor}

\begin{proof}
By Lemma \ref{new-lemma-1}, 
\begin{align*}
&|\cE_s(v,\phi_n-\phi)|\le \\
&c(N,s) [v]_{s,\Omega'} [\phi_n-\phi]_{s,\Omega'}+ \textbf{C}_1 \|v\|_{L^2(\Omega')} \|\phi_n-\phi\|_{L^2(\Omega')}
 + \textbf{C}_2 \|\phi_n-\phi\|_{L^1(\Omega')} \|v\|_{\cL^1_s},
\end{align*}
where $\textbf{C}_1$ and $\textbf{C}_2$ are positive constants. Thanks to the embeddings  $H^s_{loc}(\Omega)\hookrightarrow  L^2_{loc}(\Omega) \hookrightarrow L^1_{loc}(\Omega)$, we conclude that $\cE_s(v,\phi_n-\phi) \rightarrow 0$ as $n\rightarrow\infty.$
 \end{proof}

\section{Properties of radial solutions and their partial derivatives}\label{section3}%%%%%%%%%%%%%%%%%%%%%%%%%%%%%

In the following, we restrict our attention to the case $\Omega=\B$ and to bounded weak solutions of equation (\ref{eq1}). Here and in the following, we 
fix a nonlinearity $f: \R \to \R$ of class $C^1$, and we call a function $u \in \cH_0^s(\B) \cap L^\infty(\B)$ a weak solution of (\ref{eq1}) if 
$$
\cE_s(u,\phi) = \int_{\B}f(u) \phi\,dx \qquad \text{for all $\phi \in \cH_0^s(\B)$.}
$$
We note the following regularity properties for weak solutions of (\ref{eq1}). For this we consider the distance function to the boundary 
$$ 
\delta: \overline \B \to \R, \qquad \delta(x)=\dist(x,\partial \B)= 1-|x|.
$$

\begin{prop}(cf. \cite{Ls07,Ros14,FS-2019,Gr14})\\
\label{reg-prop-solutions}
Let $u \in \cH_0^s(\B) \cap L^\infty(\B)$ be a weak solution of (\ref{eq1}). 
Then $u\in C^{2,s}_{loc}(\B)\cap C^s_0(\ov \B)$. Moreover, 
\begin{equation}
  \label{eq:def-psi}
\psi:=\frac{u}{\delta^s} \in C^{\alpha}(\ov\B)\qquad \text{for some $\alpha \in (0,1)$,}  
\end{equation}
and the following properties hold with some constant $c>0$: 
\begin{enumerate}
\item[(i)] $| \nabla u (x)  | \leq c \delta^{s-1}(x)$ for all $x\in \B$.
\item[(ii)] $| \nabla \psi (x)  | \leq c \delta^{\alpha-1}(x)$ for all $x\in \B$.
\item[(iii)] For every $x_0 \in \partial \B$, we have 
$\lim \limits_{x \to x_0}\delta^{1-s}(x) \partial_r\, u(x) = - s \psi(x_0)$, where $\partial_r u(x)= \nabla u(x) \cdot \frac{x}{|x|}$ denotes the radial derivative of $u$ at $x$.
\item[(iv)] If $s \in (\frac{1}{2},1)$, then $\psi \in C^1(\ov \B)$. 
\end{enumerate}
\end{prop}

\begin{proof}
Since $u \in L^\infty(\B)$ and $f$ is of class $C^1$, we have $f(u(\cdot)) \in L^\infty(\B)$. Hence the regularity theory for the fractional Dirichlet-Possion problem developed in \cite{Ros14} shows that $u \in C^s_0(\B)$, and that (i) holds. It is also shown in \cite{Ros14} that $\psi:=\frac{u}{\delta^s} \in C^{\alpha}(\ov\B)$ for some $\alpha \in (0,1)$. Moreover, (ii) and (iii) are proved in \cite{FS-2019}.\\
Finally, noting that $f(u(\cdot)) \in C^s(\B)$ since $u \in C^s_0(\B)$, it follows from interior regularity (see e.g. \cite{Ls07}) that $u \in C^{2,s}_{loc}(\B)$. Moreover, if $s \in (\frac{1}{2},1)$ we have $\psi \in C^{2s}(\ov \B) \subset C^1(\ov \B)$ by \cite[Theorem 2.2]{Gr14}.
\end{proof}

The regularity estimates above allow to apply the following simple integration by parts formula to weak solutions of (\ref{eq1}).

\begin{lemma}
\label{lm-int-by-part}
Let $u\in C^0(\ov \B)\cap C^1_{loc}(\B)$ be a function satisfying $u \equiv 0$ on $\partial \B$ and $|\nabla u| \in L^1(\B)$. Then
\begin{equation}
  \label{eq:int-by-parts}
\int_{\B} (\partial_j u) \phi\ dx = -\int_{\B}  u \partial_j \phi\ dx \qquad \text{ for $\phi \in C^1(\overline \B)$, $j = 1,\dots,N$.}
\end{equation}
\end{lemma}
\begin{proof}
Let $\phi \in C^1(\overline \B)$, and let $\Omega_n:=B_{1-\frac{1}{n}}(0) \subset \B$ for $n \in \N$. Then $u\in C^1(\overline{\Omega}_n)$ for $n \in \N$ since $u \in C^1_{loc}(\B)$. Integrating by parts over $\Omega_n$ and using a change of variables, we find that
$$
\int_{\Omega_n}\bigl((\partial_j u)\phi + u\partial_j\phi\bigr)\ dx = 
\int_{\partial\Omega_n} u\phi\nu_j\ d\sigma =(1-\frac{1}{n})^{N-1}\int_{\partial\B}u((1-\frac{1}{n})\sigma)\phi((1-\frac{1}{n})\sigma)\nu_j\ d\sigma,
$$
	where $\nu_j$ is the $j$-th component of the unit outward normal to $\partial\B$ at $x$. Since $u\in C^0(\overline{\B})$, $u=0$ on $\partial\B$, $\Omega_n\uparrow\B$ and $\phi\in C^{1}(\overline{\B})$, we can apply the Lebesgue dominated convergence theorem to both sides of the equation above to deduce \eqref{eq:int-by-parts}. 
\end{proof}

In the following, we fix a {\em radial} solution $u \in \cH^s_0(\B) \cap L^\infty(\B)$ of (\ref{eq1}), and we consider the function $\psi$ defined in (\ref{eq:def-psi}) which is also radial. Hence we write 
\begin{equation}
  \label{eq:def-psi-0}
\text{$\psi(x)= \psi_0(r)$ for $r = |x|$ with a function $\psi_0: [0,1] \to \R$}  
\end{equation}
which is of class $C^\alpha$ for some $\alpha>0$ by Proposition~\ref{reg-prop-solutions}. 
Moreover, by Proposition~\ref{reg-prop-solutions} we have 
\begin{equation}
  \label{eq:d-r-equation}
\psi_0(1) = \lim\limits_{|x|\rightarrow 1} 
\frac{u(|x|)}{(1-|x|)^s}= -\frac{1}{s}\lim\limits_{|x| \rightarrow1}(1-|x|)^{1-s} \partial_r\, u(x).  
\end{equation}
By the Pohozaev type identity given in \cite[Theorem 1.1]{Ros14-1}, this value also satisfies 
\begin{equation}
  \label{eq:pohozaev}
\psi_0^2(1) = \frac{1}{|S^{N-1}|\Gamma(1+s)^2} \int_{\B} \Bigl[(2s-N) u f(u) + 2N F (u)\Bigr]dx.
\end{equation}
Here $F: \R \to \R$ is given by $F(t) = \int_0^t f(\tau)\ d\tau$.\\ 

The aim of this section is to construct test functions related to partial derivatives of $u$, which allow to estimate Dirichlet eigenvalues of the 
linearized operator
\begin{equation}\label{O2}
L:= (-\Delta)^s - f'(u).
\end{equation}
For $j \in \{1,\dots,N\}$, we consider the partial derivatives of $u$ given by  
$$
v^j: \R^N \to \R, \qquad v^j(x)= 
\left\{
  \begin{aligned}
&\partial_j u (x)=\frac{\partial u}{\partial x_{j}}(x),&&\qquad x \in \B,\\
&0,&&\qquad x \in \R^N \setminus \B,    
  \end{aligned}
\right.
\qquad j=1,\dots,N.
$$
From Proposition~\ref{reg-prop-solutions}, it then follows that 
\begin{equation}
  \label{eq:prop-v-j}
v^j \in \cL^1_s \cap H^s_{loc}(\B) \qquad \text{for $j \in \{1,\dots,N\}$.}
\end{equation}
Hence  $\cE_s(v^j,\phi)$ is well defined for every $\phi \in \cH^s_0(\B)$ with compact support by Lemma~\ref{new-lemma-1}. We have the following key lemma.

\begin{lemma}\label{l2-new}
For any $j\in \{1, \dots, N\}$, we have $Lv^j =(-\Delta)^s v^j -f'(u)v^j =0$ in distributional sense in $\B$, i.e. 
\begin{equation}
  \label{eq:distributional-sense}
\int_{\B}v^j (-\Delta)^s \phi\ dx = \cE_s(v^j,\phi)=
 \int_{\B} f'(u)v^j \phi\ dx \qquad \text{for all $\phi \in \cC^\infty_c(\B)$.}
\end{equation}
Moreover, if $\phi \in \cH^s_0(\B)$ has compact support in $\B$, then we have 
\begin{equation}
  \label{eq:weak-formulation}
\cE_{s}(v^j,\phi) =  \int_{\B} f'(u)v^j \phi\ dx.
\end{equation}
Furthermore, if $v^j \in \cH^s_0(\B)$, then (\ref{eq:weak-formulation}) is true for all $\phi \in \cH^s_0(\B)$. 
\end{lemma} 

\begin{proof} 
Since $u \in C^{2,s}_{loc}(\B)$ by Proposition~\ref{reg-prop-solutions}, we have $v^j \in C^{1,s}_{loc}(\B) \subset H^s_{loc}(\B)$. Let $\phi \in \cC^\infty_c(\B) \subset \cC^\infty_c(\R^N)$. Then 
$$
\partial_j \phi \in \cC^\infty_c(\B),\quad  (-\Delta)^s \phi \in C^\infty(\R^N),\qquad \text{and}\qquad \partial_j (-\Delta)^s \phi = (-\Delta)^s \partial_j \phi \quad \text{on $\R^N$.}
$$
Consequently, since $u$ satisfies the assumptions of Lemma~\ref{lm-int-by-part}, (\ref{eq:int-by-parts}) implies that  
\begin{align*}
\int_{\B}v^j(-\Delta)^s\phi\ dx&=-\int_{\B}u \partial_j (-\Delta)^s\phi\ dx= -\int_{\B}u  (-\Delta)^s\partial_j \phi\ dx\\
&= -\cE_s(u,\partial_j \phi) = -\int_{\B} f(u)\partial_j \phi\ dx = \int_{\B} \partial_j f(u) \phi\ dx = \int_{\B} f'(u)v^j \phi\ dx.
\end{align*}
Hence $v^j$ solves $Lv^j =(-\Delta)^s v^j -f'(u)v^j =0$ in distributional sense. Next we show that 
\begin{equation}
  \label{eq:next-step}
\cE_s(v^j,\phi) = \int_{\B}f'(u)v^j \phi\ dx \qquad \text{for all $\phi \in \cC^\infty_c(\B)$.}
\end{equation}
Since $v^j \in \cL^1_s \cap H^s_{loc}(\B)$, the integral 
$$
\int_{\R^N}\int_{\R^N}\frac{\big| v^j(x)-v^j(y)\big|\big |\phi(x)-\phi(y)\big|}{|x-y|^{N+2s}}\ dxdy
$$
exists by Lemma~\ref{new-lemma-1}, and therefore we have, by Lebesgue's Theorem, 
\begin{align*}
\cE_s(v^j,\phi)&= \frac{c(N,s)}{2} 
\lim_{\eps \to 0} \int_{\R^N} \int_{|x-y|\ge \eps}\frac{\big(v^j(x)-v^j(y)\big)\big(\phi(x)-\phi(y)\big)}{|x-y|^{N+2s}}\ dxdy\\
&=c(N,s) \lim_{\eps \to 0} \displaystyle \int_{\R^N} v^j(x) \int_{\R^N \setminus B_\eps(x)} 
 \frac{\phi(x)-\phi(y)}{|x-y|^{N+2s}}\ dy dx\\
&=c(N,s)  \displaystyle \int_{\R^N} v^j(x)\; \lim_{\eps \to 0} \int_{\R^N \setminus B_\eps(x)} 
 \frac{\phi(x)-\phi(y)}{|x-y|^{N+2s}}\ dy dx\\
&=\int_{\R^N} v^j (-\Delta)^s \phi\ dx = \int_{\B} v^j (-\Delta)^s \phi\ dx = \int_{\B} f'(u)v^j \phi \ dx.
\end{align*}
Next, let $\phi \in \cH^s_0(\B)$ with compact support in $\B$, and choose an open subset $\Omega' \subset \subset \B$ such that 
$\supp\, \phi \subset \Omega'$. By definition of $\cH^s_0(\Omega')$, there exists a sequence $(\phi_n)_n$ in $\cC^\infty_c(\Omega') \subset \cC^\infty_c(\B)$ 
with $\phi_n\rightarrow\phi~~\text{in}~~\cH^s_0(\Omega')$, hence also $\phi_n\rightarrow\phi~~\text{in}~~\cH^s_0(\B)$.  Then Corollary~\ref{New-cor} and (\ref{eq:next-step}) imply that 
\begin{equation}
  \label{eq:double-use}
\cE_s(v^j,\phi) = \lim_{n \to \infty}\cE_s(v^j,\phi_n)=  \lim_{n \to \infty}\int_{\B}f'(u)v^j \phi_n\ dx = 
\int_{\B}f'(u)v^j \phi\ dx,
\end{equation}
and thus (\ref{eq:weak-formulation}) holds. 

Finally, assume that $v^j \in \cH^s_0(\B)$, let $\phi \in \cH^s_0(\B)$, and let   
$(\phi_n)_n$ be a sequence in $\cC^\infty_c(\B)$ with $\phi_n \to \phi$ in $\cH^s_0(\B)$. Then (\ref{eq:double-use}) holds again by the continuity of the quadratic form $\cE_s$ on $\cH^s_0(\B)$, as claimed.
\end{proof}

We now have all the tools to build suitable test functions from partial derivatives in order to estimate the Morse index of $u$ as a solution of (\ref{eq1}). As remarked before, the construction is inspired by \cite{aftalion2004qualitative}.

\begin{defi}
\label{test-functions-built}
Let $\psi_0$ be the function defined in (\ref{eq:def-psi-0}).
For $j= 1,\dots,N$, we define the open half spaces 
\begin{equation}
  \label{eq:def-half-spaces}
H^j_\pm := \{x \in \R^N\::\: \pm x_j>0\}
\end{equation}
and the functions $d_j: \R^N \to \R$ by 
$$
d_j:= \left \{
  \begin{aligned}
  &(v^j)^+\,1_{H^j_+} -(v^j)^-\,1_{H^j_-} && \qquad \text{if $\psi_0(1)\ge 0$;}\\
   &(v^j)^+\,1_{H^j_-} -(v^j)^-\,1_{H^j_+}&& \qquad \text{if $\psi_0(1)<0$.}
  \end{aligned}
\right.
$$
\end{defi}
We note that, for $j=1,\dots,N$, the function $d_j$ is odd with respect to the reflection 
\begin{equation*}
\sigma_j: \R^N \to \R^N, \qquad x = (x_1,\dots,x_j,\dots, x_N)\mapsto \sigma_j(x) = (x_1,\dots,-x_j,\dots, x_N)
\end{equation*}
at the hyperplane $\{x_j = 0\}$ since the function $v^j$ is odd. 

\begin{lemma}\label{l3-new-prelim}
$d_j \in H^s_{loc}(\B)$ for $j = 1,\dots,N$. 
\end{lemma}

\begin{proof}
By definition of $d_j$, it suffices to show that 
\begin{equation}
  \label{eq:v-j-pm-reg}
(v^j)^\pm\, 1_{H_\pm^j}  \;\, \in \;H^s_{loc}(\B). 
\end{equation}
We only consider the function $(v^j)^+ \,1_{H^j_+}$, the proof for the other functions is essentially the same. As noted in (\ref{eq:prop-v-j}), we have $v^j \in H^s_{loc}(\B)$, and therefore also $(v^j)^+ \in H^s_{loc}(\B)$ by a standard estimate. To abbreviate, we now put $\chi=1_{H^j_+}$, $v:=(v^j)^+$, and we let $\Omega' \subset \subset \B$ be an open subset of $\B$. Making $\Omega'$ larger if necessary, we may assume that $\Omega'$ is symmetric with respect to the reflection $\sigma_j$. To show that $v\chi\in H^s_{loc}(\Omega')$, we write 
\begin{align*}
[v\chi]_{s,\Omega'}^2&= [v]_{s,\Omega' \cap H^j_+}^2 + \int_{\Omega' \cap H^j_+}|v(x)|^2 \int_{\Omega'\cap H^j_-}|x-y|^{-N-2s}\ dydx\\
		&\leq [v]_{s, \Omega'}^2+ \int_{\Omega' \cap H^j_+}|v(x)|^2
\int_{\{y \in \R^N, |y-x| \ge |x_j|\}}|x-y|^{-N-2s} dydx\\
		&= [v]_{s,\Omega'}^2 + \int_{\Omega' \cap H^j_+}|v(x)|^2
\int_{\{z \in \R^N, |z| \ge |x_j|\}}|z|^{-N-2s} dz dx\\
&= [v]_{s, \Omega'}^2+ \frac{|S^{N-1}|}{2s} \int_{\Omega' \cap H^j_+}|v(x)|^2 |x_j|^{-2s} dx.
		\end{align*}
Since $v = (v^j)^+ \in C^s_{loc}(\B)$ by Proposition~\ref{reg-prop-solutions} and $v \equiv 0$ on $\{x_j=0\}$, we have $|v(x)|\le C|x_j|^{s}$ for $x \in \Omega'\cap H^j_+$. Therefore, the latter integral is finite, and $(v^j)^+\,1_{H^j_+}= v\chi\in H^s_{loc}(\B)$.
\end{proof}

The next lemma is of key importance for the proof of Theorem~\ref{T}. 
\begin{lemma}\label{l3-new}
Let $j= 1,\dots,N$. 
\begin{itemize}
\item[(i)] If $\psi_0(1) \not = 0$, we have $d_j \in \cH^s_0(\B)$, and $d_j$ has compact support in $\B$.
\item[(ii)] If $s \in (\frac{1}{2},1)$ and $\psi_0(1) = 0$, then we have $v^j \in \cH^s_0(\B)$ and $d_j \in \cH^s_0(\B)$.
\end{itemize}
\end{lemma}

\begin{proof}
(i) By Lemma~\ref{compact-support-H-s-0} and Lemma~\ref{l3-new-prelim}, it suffices to show that $d_j$ has compact support in $\B$. We now distinguish the cases $\psi_0(1)>0$ and $\psi_0(1)<0$. 

If $\psi_0(1)>0$, we have $\partial_ru(x)\leq 0$ in $\B\setminus B_{r_*}(0)$ for some $r_* \in (0,1)$ by (\ref{eq:d-r-equation}), and therefore 
$$
v^j(x)= \partial_j u(x)=\frac{x_{j}}{|x|} \partial_r u(x) \le 0 \qquad \text{for $x \in \B\setminus B_{r_*}(0)$ with $x_j \ge 0$.}
$$
Consequently, $d_j(x) = (v^j)^+(x)=0$ for $x \in \B\setminus B_{r_*}(0)$ with $x_j \ge 0$. Since $d_j$ is odd with respect to the reflection $\sigma_j$ it follows that $\supp\, d_j \subset \overline{B_{r_*}(0)}$, so $d_j$ is compactly supported in $\B$.

If $\psi_0(1)<0$, we have $\partial_ru(x)\ge 0$ in $\B\setminus B_{r_*}(0)$ for some $r_* \in (0,1)$ by (\ref{eq:d-r-equation}), which in this case, similarly as above, implies that $d_j(x) = -(v^j)^-(x)=0$ for $x \in \B\setminus B_{r_*}(0)$ with $x_j \ge 0$. Again we conclude that $d_j$ is compactly supported in $\B$ since it is odd with respect to the reflection $\sigma_j$.\\[0.1cm]
(ii) Since $s \in (\frac{1}{2},1)$, it follows from Proposition~\ref{reg-prop-solutions}(iv) that 
$\psi \in C^1(\overline \B)$ and therefore $\psi_0 \in C^1([0,1])$, whereas 
$\psi_0(1)=0$ by assumption.  Consequently, $ \psi(x)\delta^{s-1}(x) \to 0$ as $ |x| \to 1$, and therefore 
$$
\nabla u(x)= \delta^s(x) \nabla \psi(x) + s \psi(x)\delta^{s-1}(x)\nabla \delta(x)\to 0 \qquad \text{as $|x| \to 1$.}
$$
It thus follows that $u \in C^1(\R^N)$ with $u \equiv 0$ on $\R^N \setminus \B$, and therefore $v^{j}\in C^{0}(\R^{N})$ with $v^{j} \equiv 0$ in $\R^{N}\setminus\B.$ To see that $v^j \in \cH^s_0(\B)$, we shall use Proposition~\ref{prop1} as follows: Since the function $f'(u)v^j$ is continuous and therefore bounded in $\overline \B$, there exists a unique weak solution $w \in\cH^s_0(\B)$ to the Poisson problem
\begin{equation}
	(-\Delta)^sw  = f'(u)v^j \quad \text{in $\B$},\qquad \qquad w = 0 \quad \text{in $\R^{N}\setminus \B$}
	\end{equation}
which satisfies $w \in C^s_0(\B)$ by \cite[Proposition 1.1]{Ros14}. By setting $V:= w - v^j$, it follows  that $V\in C^{0}(\R^N)$ with $V \equiv 0$ in $\R^{N}\setminus\B.$
 Moreover, by Lemma~\ref{l2-new} the function $V$ satisfies the equation 
$(-\Delta)^s {V} = 0$ in $\B$ in the sense of distributions. Since $V$ is continuous, Proposition \ref{prop1} -- applied to $\pm V$ -- implies that $V\equiv 0$  in $\R^N$, i.e., 
\begin{equation}
  \label{eq:conclusion-proof-ii}
v^{j}=w \in\cH^s_0(\B) \cap C^s_0(\B). 
\end{equation}
By a similar argument as in the proof of Lemma~\ref{l3-new-prelim}, we will now see that $d_j \in \cH^s_0(\B)$. For the convenience of the reader, we give the details. It is clearly sufficient to show that
\begin{equation}
  \label{eq:v-j-pm-reg1}
(v^j)^\pm \,1_{H^j_\pm} \:\, \in \;\cH^s_{0}(\B). 
\end{equation}
We only consider the function $(v^j)^+ \,1_{H^j_+}$, the proof for the other functions is the same. Since $v^j \in \cH^s_0(\B)$, we also have $(v^j)^\pm \in \cH^s_0(\B)$ by a standard estimate. To abbreviate, we now put $\chi=1_{H^j_+}$ and $v:=(v^j)^+$.
To show that $v\chi\in \cH^s_0(\B)$, we note that $v\chi \equiv 0$ in $\R^N \setminus \cB$, and we estimate 
\begin{align*}
[v\chi]_{s,\R^N}^2&= [v]_{s,H^{j}_+}^2 + \int_{H^j_+ \cap \B}|v(x)|^2 \int_{H^j_-}|x-y|^{-N-2s}\ dydx\\
		&\leq [v]_{s,\R^N}^2+ \int_{H^j_+ \cap \B}|v(x)|^2
\int_{\{z \in \R^N, |z| \ge |x_j|\}}|z|^{-N-2s} dz dx\\
&= [v]_{s,\R^N}^2 + \frac{|S^{N-1}|}{2s} \int_{H^j_+ \cap \B}|v(x)|^2 |x_j|^{-2s} dx.
		\end{align*}
Since $v = (v^j)^+ \in C^s(\ov \B)$ by (\ref{eq:conclusion-proof-ii}) and $v \equiv 0$ on $\{x_j=0\}$, we have $|v(x)|\le C|x_j|^{s}$ for $x \in H^j_+ \cap \B$. Therefore, the latter integral is finite, and $(v^j)^+ \,1_{H^j_+}= v\chi\in \cH^s_0(\B)$.
\end{proof}

\begin{cor}
\label{v-j-d-k-energy}  
If $\psi_0(1)\not = 0$ or $s \in (\frac{1}{2},1)$, then the values $\cE_{s}(d_j,d_k)$ and $\cE_{s}(v^j,d_k)$ are well-defined and satisfy
\begin{equation*}
\cE_{s}(v^j,d_k) =   \int_{\B} f'(u)v^j d_k\ dx \qquad \text{for $j,k= 1,\dots,N$.} 
\end{equation*}
\end{cor}

\begin{proof}
This follows from Lemma~\ref{new-lemma-1}, Lemma~\ref{l2-new} and Lemma~\ref{l3-new}.   
\end{proof}

\section{Proof of Theorem \ref{T}}\label{section4}%%%%%%%%%%%%%%%%%%%%%%%%%%%%%

In this section we complete the proof of Theorem \ref{T}. As before, we consider a fixed radial weak solution $u \in \cH^s_0(\B) \cap L^\infty(\B)$ of (\ref{eq1}), and we will continue using the notation related to $u$ as introduced in Section~\ref{section3}. Moreover, in accordance with the assumptions of Theorem \ref{T}, we assume that $u$ changes sign, which implies that 
\begin{equation}
  \label{eq:nontrival-partial}
(v^j)^\pm \,1_{H^j_+}\not \equiv 0  \quad \text{and}\quad (v^j)^\pm \,1_{H^j_-}
\not \equiv 0 \qquad \text{for $j = 1,\dots,N$,} 
\end{equation}
where the half spaces $H^j_\pm$ are defined in (\ref{eq:def-half-spaces}). We first note that, under the assumptions of Theorem \ref{T}, we have 
\begin{equation}
  \label{eq:alternative-main-T-consequence}
\psi_0(1)\not = 0 \qquad \text{or}\qquad s \in (\frac{1}{2},1).
\end{equation}
Indeed, if $s \in (0,\frac{1}{2}]$, then $\psi_0^2(1)>0$ by (\ref{eq:add-cond}) and (\ref{eq:pohozaev}). 

Next we recall that the $n$-th Dirichlet eigenvalue $\lambda_{n,L}$ of the linearized operator $L$ defined in (\ref{O2}) admits the variational characterization 
\begin{equation}
\label{eigenvalue-characterization}
\displaystyle\lambda_{n,L} = \min_{V \in \cV_n}\max_{v\in S_V} \cE_{s,L}(v,v)
\end{equation}
where 
\begin{equation}\label{e1}
(v,w) \mapsto \cE_{s,L}(v,w): = \cE_s(v,w) -\int_{\B}f'(u)vw\ dx
\end{equation}
is the bilinear form associated to $L$, $\cV_n$ denotes the family of $n$-dimensional subspaces of $\cH^s_0(\B)$ and $S_V:= \{v \in V\::\: \|v\|_{L^2(\B)}=1\}$ for $V \in \cV_n$. 

To estimate $\lambda_{n,L}$ from above, we wish to build test function spaces $V$ by using the functions $d_j$ introduced in Definition~\ref{test-functions-built}. By Lemma~\ref{l3-new} and (\ref{eq:alternative-main-T-consequence}), we have $d_j \in \cH^s_0(\Omega)$ for $j=1,\dots,N$. Moreover, as a consequence of Corollary~\ref{v-j-d-k-energy}, the values $\cE_{s}(v^j,d_k)$ are well-defined and satisfy
\begin{equation}
  \label{eq:weak-formulation-d-j-consequence}
\cE_{s,L}(v^j,d_k) =  0 \qquad \text{for $j,k= 1,\dots,N$.} 
\end{equation}

We need the following key inequality.

\begin{lemma}\label{l4}
For $j\in\{1,\dots, N\}$ we have $\cE_{s,L}(d_j,d_j)< 0.$
\end{lemma}
\begin{proof} 
To simplify notation, we put $k(z)= c(N,s)|z|^{-N-2s}$ for $z \in \R^N \setminus \{0\}$. Since $v^jd_j = d^2_j$ in $\R^N$ by definition of $d_j$ and therefore 
$$
\int_{\B}f'(u)v^jd_j\ dx=\int_{\B}f'(u)d_j^2\ dx,
$$
we have, by (\ref{eq:weak-formulation-d-j-consequence}), 
\begin{align*}
&\cE_{s,L}(d_j,d_j) = \cE_{s,L}(d_j-v^j,d_j)\\
&\quad =\frac{1}{2}\int_{\R^N}\int_{\R^N}\bigg(\bigl(d_j(x)-v^j(x)-(d_j(y)-v^j(y))\bigr)(d_j(x)-d_j(y))\bigg)k(x-y)dxdy\\
&\quad =\frac{1}{2}\int_{\R^N}\int_{\R^N}\bigg(v^j(x)d_j(y)+v^j(y)d_j(x)-2d_j(x)d_j(y)\bigg) k(x-y)\ dxdy
\end{align*}
In the following, we put 
$$
\ell_j(x,y):= k(x-y)-k(\sigma_j(x)-y) \qquad \text{for $x,y \in \R^N,\: x \not=y.$} 
$$
Using the oddness of the functions $v^{j}$ and $d_j$ with respect to the reflection $\sigma_j$, we deduce that 
\begin{align}
&\cE_{s,L}(d_j,d_j) =
\frac{1}{2}\int_{\R^N}\int_{H_+^j}\bigg(v^j(x)d_j(y)+v^j(y)d_j(x)-2d_j(x)d_j(y)\bigg) \ell_j(x,y)\ dxdy \nonumber\\
 &= \frac{1}{2}
\int_{H_+^j}\int_{H_+^j}\bigg(v^j(x)d_j(y)+v^j(y)d_j(x)-2d_j(x)d_j(y)\bigg)\bigl(\ell_j(x,y)-\ell_j(x,\sigma_j(y))\bigr)\ dxdy \nonumber\\
&= \int_{H_+^j}\int_{H_+^j}\bigg(v^j(x)d_j(y)+v^j(y)d_j(x)-2d_j(x)d_j(y)\bigg)\ell_j(x,y) dxdy. \label{d-j-d-jcomputation-1}
\end{align}
Here we used in the last step that  
$$
k(\sigma_j(x)-\sigma_j(y))=k(x-y)\qquad \text{and}\qquad k(\sigma_j(x)-y)=k(x-\sigma_j(y))
$$
for $x, y\in \R^N$, $x \not = y$ and therefore 
\begin{align*}
\ell_j(x,y)-\ell_j(x,\sigma_j(y)) &= k(x-y)-k(\sigma_j(x)-y)-\bigl(k(x-\sigma_j(y))-k(\sigma_j(x)-\sigma_j(y))\bigr)\\
&= 2 \ell_j(x,y).
\end{align*}
Next, we note that
\begin{equation}
\label{d-j-d-jcomputation-2}  
\ell_j(x,y)= k(x-y)-k(\sigma_j(x)-y)>0 \qquad \text{for $x,y \in H_+^j$.}
\end{equation}
Moreover, we claim that the function
\begin{align*}
  (x,y) \mapsto h_j(x,y)&=v^j(x)d_j(y)+v^j(y)d_j(x)-2d_j(x)d_j(y)\\
  &= (v^j(x)-d_j(x)) d_j(y) +(v^j(y)-d^j(y))d_j(x)
\end{align*}
satisfies 
\begin{equation}
\label{d-j-d-jcomputation-3}  
h_j \le 0 \quad \text{and}\quad h_j \not \equiv 0 \qquad \text{on $H_+^j \times H_+^j$.}
\end{equation}
Indeed, if $\psi_0(1) \ge 0$, we have $d_j = (v^{j})^+$ and therefore $v^j-d_j = -(v^{j})^-$ on $H_+^j$. Hence (\ref{d-j-d-jcomputation-3}) follows from (\ref{eq:nontrival-partial}). Moreover, if $\psi_0(1)< 0$, we have $d_j = -(v^{j})^-$ and therefore $v^j-d_j = (v^{j})^+$ on $H_+^j$. Again (\ref{d-j-d-jcomputation-3}) follows from (\ref{eq:nontrival-partial}). The claim now follows by combining (\ref{d-j-d-jcomputation-1}), (\ref{d-j-d-jcomputation-2}) and (\ref{d-j-d-jcomputation-3}).
\end{proof}

\begin{lemma}\label{l6}
Let $\alpha = (\alpha_1,\dots,\alpha_N) \in \R^N$ and $\displaystyle d = \sum_{j=1}^{N}\alpha_jd_j$. Then we have 
 $$
\cE_{s,L}(d,d) = \displaystyle  \sum_{j=1}^{N}\alpha^2_j\cE_{L}(d_j,d_j)\le 0.
$$
Moreover, 
\begin{equation}
  \label{eq:lin-combination-negative}
\cE_{s,L}(d,d)<0 \qquad \text{if and only if}\qquad \alpha \not = 0, 
\end{equation}
and therefore the functions $d_1,\dots,d_N$ are linearly independent. 
\end{lemma}

\begin{proof}
We first note that 
\begin{equation}
\label{eq:l5}
\cE_{s,L}(d_j,d_k)= 0 \qquad \text{for $j, k \in \{1, \dots , N\}$, $j\neq k$.}
\end{equation}
Indeed, since $u$ is radially symmetric, the function $d_j$ is odd with respect to the reflection $\sigma_j$ and even with respect to the reflection $\sigma_k$ for $k \not = j$. Hence, by a change of variable, 
\begin{align*}
&\cE_{s,L}(d_j, d_k)=\frac{c(N,s)}{2}\int_{\R^N}\int_{\R^N}\frac{\Big(d_j(\sigma_j(x))-d_j(\sigma_j(y))\Big)\Big(d_k(\sigma_j(x)) - d_k(\sigma_j(y))\Big)}{|\sigma_j(x)-\sigma_j(y)|^{N+2s}}\ dxdy\\
&\hspace{6cm} -\int_{\B}f'(u(\sigma_j(x)))d_j(\sigma_j(x))d_k(\sigma_j(x))\ dx\\
&=\frac{c(N,s)}{2}\int_{\R^N}\int_{\R^N}\frac{\Big(d_j(y)-d_j(x)\Big)\Big(d_k(x) - d_k(y)\Big)}{|x-y|^{N+2s}}\ dxdy+ \int_{\B}f'(u(x))d_j(x)d_k(x)\ dx\\
&=-\cE_{s,L}(d_j,d_k).
\end{align*}
Hence (\ref{eq:l5}) is true. Now, for $\alpha = (\alpha_1,\dots,\alpha_N) \in \R^N$ and $d = \sum \limits_{j=1}^{N}\alpha_jd_j$, we have 
$$
\cE_{s,L}(d,d)=\sum_{j=1}^{N}\alpha_j^2\cE_{s,L}(d_j,d_j)+\sum_{\substack{j,k=1\\j\neq k}}^{N}\alpha_j\alpha_k\cE_{s,L}(d_j,d_k)= \sum_{j=1}^{N}\alpha_j^2\cE_{s,L}(d_j,d_j)\le 0
$$
by \eqref{eq:l5} and Lemma~\ref{l4}. Moreover, if $\alpha \not =0$, it follows from Lemma~\ref{l4}
that $\cE_{s,L}(d,d) <0$, which in particular implies that $d \not =0$. Consequently, the functions $d_1,\dots,d_N$ are linearly independent, as claimed.  
\end{proof}

\begin{lemma}\label{the-first-eigenfunction-of-L}
The first eigenvalue $\lambda_{1,L}$ of the operator $L = (-\Delta)^s - f'(u)$ is simple, and the corresponding eigenspace is spanned by radially symmetric eigenfunction $\phi_{1,L}$. Furthermore, 
$$
\cE_{s,L}(d_j,\phi_{1,L})=0 \quad \text{for $j=1,2, \dots, N$}\qquad \text{and}\qquad\lambda_{1,L}= \cE_{s,L}(\phi_{1,L},\phi_{1,L})<0. 
$$
\end{lemma}

\begin{proof}
The simplicity of $\lambda_{1,L}$ and the radial symmetry of $\phi_{1,L}$ are well known, but we recall the proof for the convenience of the reader. 
The variational characterization of $\lambda_{1,L}$ is given by 
$$
\lambda_{1,L} =\inf_{v\in \cH^s_0(\B)\setminus\{0\}}\frac{\cE_{s,L}(v,v)}{\|v\|^2_{L^2(\B)}} ~=~ \inf_{M}\cE_{s,L}(v,v) \quad \text{with}\quad M= \{v\in \cH^s_0(\B):\|v\|_{L^2(\B)}=1\},
$$
and the associated minimizers $\phi \in M$ are precisely the $L^2$-normalized eigenfunctions of $L$ corresponding ot $\lambda_{1,L}$, i.e., the $L^2$-normalized (weak) solutions of 
\begin{equation}
  \label{eq:weak-eigenfunction}
L\phi =  \lambda_{1,L} \phi \quad \text{in $\B$,}\qquad \phi \equiv 0 \quad \text{in $\R^N \setminus \B$.}
\end{equation}
Moreover, if $\phi \in M$ is such a minimizer, then also $|\phi| \in M$ and 
$$
\lambda_{1,L} =\cE_{s,L}(\phi,\phi)\ge \cE_{s,L}(|\phi|,|\phi|) \ge \inf_{M}\cE_{s,L}(v,v) = \lambda_{1,L},	
$$
which implies that $|\phi|$ is also a minimizer and therefore a weak solution of (\ref{eq:weak-eigenfunction}). By the strong maximum principle for nonlocal operators (see e.g. \cite[p.312--313]{BB00} or \cite{JW-2019}), $|\phi|$ is strictly positive in $\B$. Consequently, every eigenfunction $\phi$ of $L$ is either strictly positive or strictly negative in $\B$. Consequently, $\lambda_{1,L}$ does not admit two $L^2$-orthogonal eigenfunctions, and therefore $\lambda_{1,L}$ is simple. 

Next we note that, by a simple change of variable, if $\phi$ is an eigenfunction of $L$ corresponding to $\lambda_{1,L}$, then also $\phi \circ \cR$ is an eigenfunction for every rotation $\cR \in O(N)$. Consequently, the simplicity of $\lambda_{1,L}$ implies that the associated eigenspace is spanned by a radially symmetric eigenfunction $\phi_{1,L}$.

Next, using the radially symmetry of $u$ and $\phi_{1,L}$ and the oddness of $d_j$ with respect to the reflection $\sigma_j$, we find, by a change of variable, that 
	\begin{align*}
	&\cE_{s,L}(d_j,\phi_{1,L})=\frac{c(N,s)}{2}\int_{\R^N}\int_{\R^N}\frac{(d_j(\sigma_j(x))-d_j(\sigma_j(y)))(\phi_{1,L}(\sigma_j(x))-\phi_{1,L}(\sigma_j(x)))}{|x-y|^{N+2s}}dxdy\\
	&\hspace{6cm}-\int_{\B}f'(u(\sigma_j(x)))d_j(\sigma_j(x))\phi_{1,L}(\sigma_j(x))\ dx\\
&=\frac{c(N,s)}{2}\int_{\R^N}\int_{\R^N}\frac{(d_j(y)-d_j(x))(\phi_{1,L}(x)-\phi_{1,L}(y))}{|x-y|^{N+2s}}\ dxdy+\int_{\B}f'(u(x))d_j(x)\phi_{1,L}(x)dx\\
&=-\cE_{s,L}(d_j,\phi_{1,L})
	\end{align*}
and therefore $\cE_{s,L}(d_j,\phi_{1,L})=0$ for $j=1,\dots,N$. 
Finally, by Lemma~\ref{l4} and the variational characterization of $\lambda_{1,L}$, we have $\lambda_{1,L}= \cE_{s,L}(\phi_{1,L},\phi_{1,L})<0$, as claimed.
\end{proof}

\begin{proof}[Proof of Theorem \ref{T}(completed)]
Let $\phi_{1,L} \in \cH^s_0(\B)$ be an eigenfunction of $L$ corresponding to the first eigenvalue $\lambda_{1,L}$ as given in Lemma~\ref{the-first-eigenfunction-of-L}. We consider the subspace $V = \text{ span}\{\phi_{1,L}, d_1, \dots,d_N \}$. For $\alpha \in \R^{N+1} \setminus \{0\}$ and 
$d = \alpha_0 \phi_{1,L} + \sum \limits_{j=1}^{N}\alpha_jd_j \in V$, we then have, by Lemma~\ref{l6} and Lemma~\ref{the-first-eigenfunction-of-L},   
$$
\cE_{s,L}(d,d) = \alpha_0^2\,\cE_{s,L}(\phi_{1,L},\phi_{1,L}) +  \cE_{s,L}(\sum_{j=1}^{N}\alpha_jd_j,\sum_{j=1}^{N}\alpha_jd_j)<0.
$$
In particular, it follows that the functions $\phi_{1,L}, d_1, \dots,d_N$ are linearly independent and therefore $V$ is $N+1$-dimensional. By (\ref{eigenvalue-characterization}) and the compactness of $S_V= \{v \in V\::\: \|v\|_{L^2(\B)}=1\}$, it then follows that $\lambda_{N+1,L}<0$, which means that $u$ has Morse index greater than or equal to $N+1\ge 2$, as claimed.
\end{proof}

\section{The linear case}
\label{alternative}

In this section we discuss the linear eigenvalue problem~(\ref{eq1-linear})  and complete the proof of Theorem~\ref{symmetry-of-higher-eigenfunctions}. In particular, we wish to recall a useful characterization of eigenvalues and eigenfunctions of 
(\ref{eq1-linear}) derived in \cite{dyda2017eigenvalues}. For this we need to consider the following radially symmetric version of (\ref{eq1-linear}) in general dimensions $d \in \N$: 
	\begin{equation}\label{eq1-linear-radial}
	\quad\left\{\begin{aligned}
		&(-\Delta)^su = \lambda u && \text{ in\ \  $\B \subset \R^d$}\\
		&u \in \cH^s_0(\B),&& \text{$u$ radially symmetric.}
	\end{aligned}\right.
	\end{equation}
In the following, we let $\lambda_{d,0} < \lambda_{d,1} \le \dots$ denote the increasing sequence of eigenvalues of this problem (counted with multiplicity).

The following characterization is essentially a reformulation of \cite[Proposition 1.1]{dyda2017eigenvalues}.

\begin{prop}
\label{sec:remarks-linear-case}
The eigenvalues of (\ref{eq1-linear}) in $\B \subset \R^N$ are of the form $\lambda= \lambda_{N+2\ell,n}$ with integers $\ell,n \ge 0$. Moreover, if 
$$
Z_\lambda:= \{(\ell,n)\::\: \lambda_{N+2\ell,n} = \lambda  \},
$$
then the eigenspace corresponding to $\lambda$ is spanned by functions of the form $u(x)= V_{\ell}(x)\phi_{N+2\ell,n}(|x|)$, where $(\ell,n) \in Z_\lambda$, $V_\ell$ is a solid harmonic polynomial of degree $\ell$ and $x \mapsto \phi_{N+2\ell,n}(|x|)$ is a (radial) eigenfunction of the problem (\ref{eq1-linear-radial}) in dimension $d=N+2\ell$ corresponding to the eigenvalue $\lambda_{N+2\ell,n}$.   
\end{prop}
Here and in the following, a solid harmonic polynomial $V$ of degree $\ell$ is a function of the form $V(x)=|x|^\ell Y(\frac{x}{|x|})$, where $Y$ is a spherical harmonic of degree $\ell$. Hence $V: \R^N \to \R$ is a homogenous polynomial of degree $\ell$ satisfying $\Delta V=0$. 

Regarding the eigenvalues $\lambda_{d,n}$ of (\ref{eq1-linear-radial}), it is also proved in \cite[Section 3]{dyda2017eigenvalues} that\begin{equation}
  \label{eq:sequence-increasing}
\text{the sequence $(\lambda_{d,0})_d$ is strictly increasing in $d \ge 1$.}  
\end{equation}
Moreover, 
\begin{equation}
  \label{eq:simplicity-consequence}
\lambda_{d,n}>\lambda_{d,0}\qquad \text{for every $d,n \ge 1$}   
\end{equation}
by the simplicity of the first eigenvalue of (\ref{eq1-linear-radial}). Consequently, the first eigenvalue $\lambda_1$ of (\ref{eq1-linear}) equals $\lambda_{N,0}$, whereas the second eigenvalue $\lambda_2$ of (\ref{eq1-linear}) is given as the minimum of $\lambda_{N+2,0}$ and $\lambda_{N,1}$.

Theorem~\ref{symmetry-of-higher-eigenfunctions} is now a direct consequence of the following result, which we will derive from Theorem~\ref{T} and from the observations above.

\begin{thm}
\label{sec:remarks-linear-case-1-thm}
We have $\lambda_{N+2,0}< \lambda_{N,1}$. Consequently, the second eigenvalue $\lambda_2$ of (\ref{eq1-linear}) is given by $\lambda_{N+2,0}$, and every corresponding eigenfunction $u$ is antisymmetric, i.e., it satisfies $u(-x)=-u(x)$ for every $x \in \B$.   
\end{thm}

\begin{proof}
Suppose by contradiction that $\lambda_2 = \lambda_{N,1} \le \lambda_{N+2,0}$. Then, noting that the only solid harmonic polynomials of degree zero are the constants, it follows from Proposition~\ref{sec:remarks-linear-case} that (\ref{eq1-linear}) admits a radially symmetric eigenfunction corresponding to $\lambda_2$. But then $u$ is a radially symmetric sign changing solution of (\ref{eq1}) with $t \mapsto f(t)= \lambda_2 t$, so it must have Morse index greater than or equal to $N+1$. This contradicts the fact that $\lambda_2$ is the second eigenvalue.\\
We thus conclude that $\lambda_2 = \lambda_{N+2,0} <\lambda_{N,1}$. Combining this inequality with (\ref{eq:sequence-increasing}) and (\ref{eq:simplicity-consequence}), we then deduce that $Z_{\lambda_2}= \{(1,0)\}$, and therefore the eigenspace corresponding to $\lambda_2$ is spanned by functions of the form $x \mapsto V_1(x) \phi_{N+2,0}(|x|)$, where $V_1$ is a solid harmonic polynomial of degree one, hence a linear function, and $x \mapsto \phi_{N+2,0}(|x|)$ is an eigenfunction of the problem (\ref{eq1-linear-radial}) in dimension $d=N+2$ corresponding to the eigenvalue $\lambda_{N+2,0}$. Since every such function is antisymmetric, the claim follows. 
\end{proof}

\bibliographystyle{amsplain}

\end{document}